\documentclass[a4paper,12pt,reqno]{amsart}

\newtheorem{proposition}{Proposition}[section]
\newtheorem{lemma}[proposition]{Lemma}
\newtheorem{corollary}[proposition]{Corollary}
\newtheorem{theorem}[proposition]{Theorem}

\theoremstyle{definition}
\newtheorem{definition}[proposition]{Definition}
\newtheorem{example}[proposition]{Example}

\newcommand{\thlabel}[1]{\label{th:#1}}
\newcommand{\thref}[1]{Theorem~\ref{th:#1}}
\newcommand{\selabel}[1]{\label{se:#1}}
\newcommand{\seref}[1]{Section~\ref{se:#1}}

\newcommand{\colabel}[1]{\label{co:#1}}

\newcommand{\delabel}[1]{\label{de:#1}}
\newcommand{\deref}[1]{Definition~\ref{de:#1}}
\newcommand{\eqlabel}[1]{\label{eq:#1}}
\newcommand{\equref}[1]{(\ref{eq:#1})}

\newcommand{\Hom}{{\rm Hom}}

\newcommand{\Cc}{\mathcal{C}}

\def\*C{{}^*\hspace*{-1pt}{\Cc}}
\def\text#1{{\rm {\rm #1}}}

\newcommand{\half}{\textstyle{\frac{1}{2}}}
\newcommand{\third}{\textstyle{\frac{1}{3}}}

\newcommand{\trl}{\triangleleft}
\newcommand{\trr}{\triangleright}
\newcommand{\ppl}{\leftharpoonup}
\newcommand{\ppr}{\rightharpoonup}

\input xy
\xyoption {all} \CompileMatrices

\usepackage{amssymb}
\usepackage{color,amssymb,graphicx,amscd,amsmath}
\usepackage[colorlinks,urlcolor=blue,linkcolor=blue,citecolor=blue]{hyperref}
\allowdisplaybreaks

\begin{document}
\title[Extending structures for Zinbiel algebras]
{Extending structures  for Zinbiel algebras}

\author[T. Zhang]{Tao Zhang}
\address{College of Mathematics and Information Science\\
Henan Normal University\\
Xinxiang 453007, PR China}
\email{zhangtao@htu.edu.cn}

\author[L. Zhang]{Ling Zhang}
\address{College of Mathematics and Information Science\\
Henan Normal University\\
Xinxiang 453007, PR China}
\email{zhanglingny@163.com}


\subjclass[2020]{17A32, 17D99}

\keywords{Zinbiel algebras, unified products,  extending structures, non-abelian cohomology, complements}


\begin{abstract}
The extending structures and unified products for Zinbiel algebras are developed.
Some special cases of unified products such as crossed products  and matched pair of Zinbiel algebras are studied.
It is proved that the extending structures can be classified by some non-abelian cohomology theory.
One dimensional flag extending structures of Zinbiel algebras are also investigated.
\end{abstract}

\maketitle

\section*{Introduction}
Zinbiel algebra is an important type of non-associative algebras related to Leibniz algebras.
It proved that the cohomology of Leibniz algebras and rack cohomology is a Zinbiel algebra, see  \cite{Lod2, Covez}.
It is also known as Tortkara algebras, pre-commutative and chronological algebras \cite{Kaw,K16}.
The classification of low dimensional Zinbiel algebras was investigated in \cite{AOK,ALO,DT,KPPV,KAM,Ni}.
According to operad theory, Zinbiel algebras are Koszul dual to Leibniz algebras.
Some other algebraic theory of Zinbiel algebras were studied in \cite{MS,N1,N2,S20,Yau}.
It has been appeared some interesting applications of Zinbiel algebras in multiple zeta values and construction of a Cartesian differential category
\cite{C1,IP}. For more recent studies of Zinbiel algebras,  see \cite{CKK,Covez,GGZ,HD}.

On the other hand,  there is an extending problem for any variety of algebra.
Given a certain variety of algebra $A$, $E$ a vector space containing $A$ as a subspace. The extending structures problem asks for the classification of all algebra structures on $E$ such that $A$ is a subalgebra of $E$.
The extending structures for groups, quantum groups, Lie algebras, associative algebras and Leibniz algebras were studied by Agore and Militaru in \cite{AM01,AM02,AM1,AM2,AM3,AM4,AM5, AM6}.
The extending structures for left-symmetric algebras, associative and Lie conformal algebras has also been studied by Y. Hong and Y. Su in  \cite{Hong1,Hong2,Hong3}.
Extending structures for Lie conformal superalgebras, 3-Lie algebras, Lie bialgebras and infinitesimal bialgebras were studied  in \cite{ZCY,Zhang1,Zhang2,Zhang3}.

In this paper, we will study  extending structures and unified products for Zinbiel algebras.
Let ${Z}$ be a Zinbiel algebra and $E$ a vector space containing ${Z}$ as a subspace.
We will describe and classify up to an isomorphism of Zinbiel algebras that stabilizes ${Z}$. We will find  the set of all Zinbiel algebra structures that can be defined on $E$  such that ${Z}$ is a subalgebra of $E$.
It is proved that the extending structures described and classified
by two dimensional non-abelian cohomology theory.

The paper is organized as follows. In \seref{unifiedprod} we
introduce the abstract construction of the \emph{unified product}
${Z} \natural V$ for Zinbiel algebras. It is associated to
a Zinbiel algebra ${Z}$, a vector space $V$ and a system
of data $\Omega({Z}, V)$ called an extending datum of ${Z}$ through
$V$.
In \seref{cazurispeciale} we show some special cases of unified products: crossed products and bicrossed products.
In \seref{exemple} we study flag extending structures of a given Zinbiel algebra ${Z}$
by $V$ in the case when $V$ is a one dimensional vector space.

Throughout this paper, all vector spaces are assumed to be over an algebraically closed field $k$ of characteristic 0.
The space of linear maps from $V$ to $W$ is denoted by $\Hom(V,W)$.
The identity map of a vector space $V$ is denoted by $id_V: V\to V$ or simply $id: V\to V$.

\section{Unified products for Zinbiel algebras}\selabel{unifiedprod}

First we recall some facts and definitions about Zinbiel algebras.

\begin{definition}
A Zinbiel algebra is a vector space ${Z}$ together with a bilinear map $\cdot  : {Z} \times {Z} \to
{Z}$ satisfying the following Zinbiel identity:
\begin{equation}
(x\cdot y)\cdot z=x\cdot(y\cdot z+z\cdot y)
\end{equation}
for all $x,y,z \in {Z}$.
\end{definition}

A homomorphism between two Zinbiel algebras $({Z},\cdot)$ and $({Z}',\cdot')$ is a linear map $\phi:{Z}\to {Z}'$ such that
$$\phi(x\cdot y)=\phi(x)\cdot'\phi(y)$$
for all $x,y\in{Z}.$  A homomorphism  is an isomorphism if it is a bijective map.

Let ${Z}$ be a Zinbiel algebra. A subspace $I\subseteq {Z}$ is called a two-sided ideal of ${Z}$ if
$x\cdot u \in I$ and $u\cdot x \in I$, for all $u\in I$ and $x \in {Z}$. For a  homomorphism $\phi:{Z}\to {Z}'$,  it is easy to see that the kernel space $\operatorname{Ker}(\phi)$ is a two-sided ideal of ${Z}$.

\begin{definition}
Let ${Z}$ be a Zinbiel algebra, $V$ a vector space. A bimodule of ${Z}$ over a vector space $V$ is a pair of linear maps \, $ \trr :{Z}\times V \to V,(x,v) \to x \trr v$ and $\trl :V\times {Z} \to V,(v,x) \to v \trl x$  such that the following conditions hold:
\begin{eqnarray}
  &&(x \cdot y) \trr v = x \trr (y \trr v+v\trl y),\\
  && (v \trl x) \trl y =v\trl(x \cdot y+y \cdot x) ,\\
  &&(x\trr v)\trl y  = x\trr (v \trl y+y \trr v),
\end{eqnarray}
for all $x,y\in {Z}$ and $v\in V.$
\end{definition}

The proof of the following Proposition \ref{prop:01} is  by easy direct computations, so we omit the details.

\begin{proposition}\label{prop:01}
Let $Z$ be a Zinbiel algebra and $(V,\triangleright)$  be a bimodule. Then the direct sum vector space $Z \oplus V$ is a Zinbiel algebra with the bilinear map defined by:
\begin{equation}
(x, u)\cdot (y, v)= (x\cdot y, x \triangleright v + u\triangleleft y)
\end{equation}
for all $ x, y \in Z, u, v \in V$. This is called the semi-direct product of $Z$ and $V$.
\end{proposition}

\begin{definition}
Let ${Z}$ be a given Zinbiel algebra, $E$ a vector space.
An extending structure of ${Z}$ through $V$ is a Zinbiel algebra on $E$
such that ${Z}$ is a subalgebra of $E$ and $V$ a complement of ${Z}$ in $E$,
which fit into the following exact sequence as vector spaces
\begin{eqnarray} \label{extencros0}
\xymatrix{ 0 \ar[r] & {Z} \ar[r]^{i} & {E}
\ar[r]^{\pi} &V \ar[r] & 0 }.
\end{eqnarray}
The extending problem is to describe and classify up to an isomorphism  the set of all Zinbiel algebra structures that can be defined on $E$  such that ${Z}$ is a Zinbiel subalgebra of $E$.
\end{definition}

\begin{definition} \delabel{echivextedn}
Let ${Z}$ be a Zinbiel algebra, $E$ a vector space such that
${Z}$ is a subalgebra of $E$ and $V$ a complement of
${Z}$ in $E$. For a linear map $\varphi: E \to E$ we
consider the diagram:
\begin{eqnarray} \eqlabel{diagrama1}
\xymatrix {& {Z} \ar[r]^{i} \ar[d]_{id} & {E}
\ar[r]^{\pi} \ar[d]^{\varphi} & V \ar[d]^{id}\\
& {Z} \ar[r]^{i} & {E}\ar[r]^{\pi } & V}
\end{eqnarray}
where $\pi : E \to V$ is the canonical projection of $E =
{Z} \oplus V$ on $V$ and $i: {Z} \to E$ is the
inclusion map. We say that $\varphi: E \to E$ \emph{stabilizes}
${Z}$ if the left square of the diagram \equref{diagrama1} is
commutative, and $\varphi: E \to E$ \emph{stabilizes}
$V$ if the right square of the diagram \equref{diagrama1} is
commutative.

Let $(E,\cdot)$ and $(E,\cdot')$ be two Zinbiel algebra structures
on $E$ both containing ${Z}$ as a subalgebra. $(E,\cdot)$ and $(E,\cdot')$ are called \emph{equivalent}, and we
denote this by $(E, \cdot) \equiv (E, \cdot')$, if
there exists a Zinbiel algebra isomorphism $\varphi: (E, \cdot)
\to (E, \cdot')$ which stabilizes ${Z}$. Denote by $Extd(E,{Z})$ the set of equivalent classes of ${Z}$ through $V$.

 $(E,\cdot)$ and $(E,\cdot')$ are called \emph{cohomologous},
and we denote this by $(E,\cdot) \approx (E, \cdot')$, if there exists a Zinbiel algebra isomorphism
$\varphi: (E, \cdot) \to (E,\cdot')$ which stabilizes ${Z}$ and co-stabilizes $V$.
Denote by $Extd'(E,{Z})$  the set of  cohomologous classes of ${Z}$ through $V$.
\end{definition}

In the following of this section we give a theoretical answer to the
extending problem by constructing two cohomological type objects which
will parameterize ${\rm Extd} \, (E, \, {Z})$ and ${\rm
Extd}' \, (E, \, {Z})$.

\begin{definition}\delabel{exdatum}
Let ${Z}$ be a Zinbiel algebra and $V$ a vector space. An
\textit{extending datum of ${Z}$ through $V$} is a system
$\Omega({Z}, V) = \bigl(\trl, \, \trr,
\, \ppl, \, \ppr, \, \omega, \, \ast  \bigl)$
consisting of six bilinear maps:
\begin{eqnarray*}
\trl : V \times {Z} \to V, \quad
&& \trr\, : {Z} \times V \to V,\quad \ppr: V \times {Z} \to {Z}, \quad\ppl \, : {Z} \times V \to {Z}, \\
&& \omega: V\times V \to {Z},\quad \ast: V\times V\to V.
\end{eqnarray*}
Let $\Omega({Z}, V) = \bigl(\trl, \,
\trr, \, \ppl, \, \ppr, \, \omega, \, \ast  \bigl)$ be an extending datum. We denote by ${Z}
\, \natural V$ the direct sum vector space ${Z}\oplus V$ together
with the the bilinear map defined by:
\begin{equation}\eqlabel{brackunif}
(x, u) \circ (y, v) := \big( x\cdot y + x \ppl v + u\ppr y + \omega(u, v),\  x \trr v + u\trl y + u \ast v \big)
\end{equation}
for all $ x,y \in {Z}$, $u,v \in V$.
The object ${Z} \natural V$ is called the \textit{unified product} of ${Z}$ and $V$ if it is a Zinbiel algebra with the the bilinear map  given by \equref{brackunif}.
\end{definition}

The following theorem provides the set of axioms that need to be
fulfilled by an extending datum $\Omega({Z}, V)$ such
that ${Z} \natural V$ is a unified product.

\begin{theorem}\thlabel{1}
Let ${Z}$ be a Zinbiel algebra, $V$ a vector space and
$\Omega({Z}, V)$ an extending datum of ${Z}$ by
$V$. Then ${Z} \natural V$ is a unified product if and
only if the following compatibility conditions hold for all $u$,
$v, w\in V$, $x$, $y$, $z \in {Z}$:
\begin{enumerate}
\item[(Z1)] $(V, \trl, \trr)$ is a $Z$-bimodule:
            \begin{eqnarray*}
            &&( x\cdot y )\trr w=x\trr( y\trr w+w\trl y),\\
            &&(x\trr v)\trl z=x\trr(v\trl z+ z\trr v),\\
            &&(u\trl y)\trl z=u\trl (y\cdot z+ z\cdot y),
            \end{eqnarray*}
\item[(Z2)]
\begin{eqnarray*}
&& (x\ppl v)\cdot y+(x \trr v)\ppr y\\
&=&x\cdot (v \ppr y+y \ppl v)+x\ppl( v\trl y + y \trr v),
\end{eqnarray*}
\item[(Z3)] $$(u \ppr x)\cdot y +( u\trl x )\ppr y = u\ppr( x \cdot  y + y\cdot  x ),$$
\item[(Z4)]
\begin{eqnarray*}
&&\omega(u, v)\cdot x+ (u \ast v )\ppr x\\
&=& u\ppr( v\ppr x + x \ppl v)+ \omega(u,v\trl x+ x \trr v),
\end{eqnarray*}
\item[(Z5)] \begin{eqnarray*}
&&(u \ast v)\trl x = u\trl( v\ppr x+x\ppl v)+u \ast( v\trl x + x \trr v ),
\end{eqnarray*}
\item[(Z6)]\begin{eqnarray*}
   &&( x\cdot y )\ppl w  =x\cdot(y \ppl w+w\ppr y)+x\ppl(y \trr w+w\trl y),
   \end{eqnarray*}
\item[(Z7)]
\begin{eqnarray*}
&&( x \ppl v)\ppl w+\omega(x \trr v,w)\\
&=&x\cdot\big(\omega(v, w)+\omega(w,v)\big)+x\ppl(v \ast w +w \ast v ),
\end{eqnarray*}
\item[(Z8)] $$(x \ppl v) \trr w+(x\trr v)\ast w = x\trr( v \ast w+w \ast v ),$$
\item[(Z9)]
\begin{eqnarray*}
&&( u\ppr x)\ppl w+\omega( u\trl x,w)\\
&=& u\ppr( x \ppl w+w\ppr x )+\omega(u, w\trl x + x\trr w),
\end{eqnarray*}
\item[(Z10)]
\begin{eqnarray*}
&&(u\ppr x) \trr w+ (u\trl x)\ast w \\
&=& u\trl(x \ppl w + w\ppr x )+u \ast(x \trr w + w\trl x),
\end{eqnarray*}
\item[(Z11)]
\begin{eqnarray*}
&&\omega(u, v)\ppl w+\omega( u \ast v ,w )\\
&=&u\ppr\big(\omega(v, w)+\omega(w,v)\big)+ \omega(u,v \ast w+ w \ast v ),
\end{eqnarray*}
\item[(Z12)]
\begin{eqnarray*}
&&\omega(u, v) \trr w+(u \ast v)\ast w \\
&=& u\trl \big(\omega(v, w)+\omega(w,v)\big)+u \ast(v \ast w+w \ast v).
\end{eqnarray*}
\end{enumerate}
\end{theorem}

\begin{proof} The proof is by direct computation of the Zinbiel identity for the the bilinear map given by \equref{brackunif}.
We check the following condition:
\begin{equation}\label{bigproduct}
\Big((x + u) \circ (y + v)\Big) \circ (z+w)=(x + u) \circ \Big( (y + v)\circ (z+w)\Big)\\
+(x + u) \circ \Big((z+w)\circ (y + v) \Big).
\end{equation}
The left hand side  is equal to
\begin{eqnarray*}
&&\Big((x + u) \circ (y + v)\Big) \circ (z+w)\\
&=&\Big[\Big( x\cdot y + x \ppl v + u\ppr y + \omega(u, v)\Big)\\
&&+\Big(x \trr v + u\trl y + u \ast v \Big)\Big] \circ (z+w)\\
&=&\Big( x\cdot y + x \ppl v + u\ppr y + \omega(u, v)\Big)\cdot z\\
&&+\Big( x\cdot y + x \ppl v + u\ppr y + \omega(u, v)\Big)\ppl w\\
&&+\Big(x \trr v + u\trl y + u \ast v \Big)\ppr z +\omega\Big((x \trr v + u\trl y + u \ast v ),w \Big)\\
&&+\Big[( x\cdot y + x \ppl v + u\ppr y + \omega(u, v)\Big] \trr w\\
&&+\Big(x \trr v + u\trl y + u \ast v \Big)\trl z+\Big(x \trr v + u\trl y + u \ast v \Big)\ast w\\
&=& (x\cdot y)\cdot z+(x\ppl v)\cdot z + (u \ppr y)\cdot z + (\omega(u, v))\cdot z+( x\cdot y )\ppl w \\
&&+( x \ppl v)\ppl w+( u\ppr y)\ppl w+ \omega(u, v)\ppl w+(x \trr v)\ppr z \\
&&+( u\trl y )\ppr z+ (u \ast v )\ppr z+\omega(x \trr v,w) +\omega( u\trl y,w) \\
&&+\omega( u \ast v ,w )+(x\cdot y)\trr w + (x \ppl v) \trr w+ (u\ppr y) \trr w \\
&&+ \omega(u, v) \trr w +(x \trr v)\trl z+(u\trl y )\trl z+(u \ast v)\trl z  \\
&&+(x \trr v )\ast w+(u\trl y)\ast w + (u \ast v )\ast w,
\end{eqnarray*}
and right hand side  is equal to
\begin{eqnarray*}
&& (x + u) \circ\Big((y + v) \circ (z+w)\Big)\\
&=& (x + u) \circ\Big[\Big( y\cdot z + y \ppl w + v\ppr z + \omega(v, w)\Big)\\
&&+\Big(y \trr w + v\trl z + v \ast w \Big)\Big]\\
&=&\Big(x\cdot ( y\cdot z + y \ppl w + v\ppr z + \omega(v, w)\Big)+x\ppl(y \trr w + v\trl z + v \ast w )\\
&&+u\ppr\Big( y\cdot z + y \ppl w + v\ppr z + \omega(v, w)\Big)+\omega(u, (y \trr w + v\trl z \\
&&+ v \ast w) \Big)+\Big(x\trr(y \trr w + v\trl z + v \ast w)\Big)+u\trl( y\cdot z + y \ppl w \\
&& + v\ppr z + \omega(v, w)\Big)+u \ast(y \trr w + v\trl z + v \ast w)\\
&=& x\cdot (y\cdot z)+x\cdot (y \ppl w)+x\cdot (v \ppr z) + x\omega(v, w)+x\ppl(y \trr w) \\
&&+x\ppl( v\trl z )+ x\ppl(v \ast w )+u\ppr( y\cdot z) + u\ppr( y \ppl w) \\
&&+ u\ppr( v\ppr z) + u\ppr\omega(v, w)+\omega(u,y \trr w) \\
&&+ \omega(u,v\trl z) + \omega(u,v \ast w) + x\trr(y \trr w)\\
&&+x\trr(v\trl z) +x\trr( v \ast w)+u\trl( y\cdot z) + u\trl(y \ppl w ) \\
&&+u\trl( v\ppr z)+u\trl \omega(v, w)+u \ast(y \trr w)\\
&&+u \ast( v\trl z) + u \ast(v \ast w),
\end{eqnarray*}
\begin{eqnarray*}
&& (x + u) \circ \Big((z+w)\circ (y + v) \Big)\\
&=&(x + u) \circ\Big[\Big( z\cdot y + z \ppl v + w\ppr y + \omega(w,v)\Big)\\
&&+\Big(z \trr v + w\trl y + w \ast v \Big)\Big]\\
&=&\Big(x\Big(z\cdot y + z \ppl v + w\ppr y + \omega(w,v)\Big)+x\ppl(z \trr v + w\trl y + w \ast v )\\
&&+u\ppr\Big(z\cdot y + z \ppl v + w\ppr y + \omega(w,v)\Big)+\omega\Big(u, (z \trr v + w\trl y \\
&&+ w \ast v )\Big)+\Big(x\trr(z \trr v + w\trl y + w \ast v )+u\trl\Big(z\cdot y + z \ppl v \\
&&+ w\ppr y + \omega(w,v)\Big) +u \ast(z \trr v + w\trl y + w \ast v) \Big)\\
&=&x\cdot (z\cdot y)+x\cdot ( z \ppl v) +x\cdot ( w\ppr y) + x\omega(w,v)\\
&&+x\ppl(z \trr v )+x\ppl( w\trl y )+ x \ppl (w \ast v )\\
&&+u \ppr(z\cdot y )+u\ppr( z \ppl v) + u \ppr(w\ppr y )\\
&&+u \ppr\Big( \omega(w,v)\Big)+\omega(u, z \trr v) +\omega(u, w\trl y) +\omega(u, w \ast v )\\
&&+x\trr(z \trr v )+x\trr(w\trl y)+x\trr(w \ast v )+u\trl(z\cdot y)+u\trl(z\ppl v)\\
&&+u\trl(w\ppr y)+u\trl\omega(w,v)+u \ast(z \trr v )\\
&&+u \ast(w\trl y)+u \ast (w \ast v).
\end{eqnarray*}
If we compare both the two sides item by item, one will find the two sides are equal to each other  if and only if $(Z1)-(Z12)$ hold. The proof is completed.
\end{proof}

For this moment, an extending structure of a  Zinbiel algebra ${Z}$
through $V$ will be viewed as a system $\Omega({Z}, V) =
\bigl(\trl, \, \trr, \, \ppl, \,
\ppr, \, \omega,\, \ast  \bigl)$ satisfying the
compatibility conditions (Z1)--(Z12). We denote by $
{\mathcal Z} ({Z}, V)$ the set of all extending
structures of ${Z}$ through $V$. Given an
extending structure $\Omega({Z},V)$, then
${Z}$ can be seen as a Zinbiel subalgebra of ${Z}\natural V$.
On the contrary, we will prove that any Zinbiel algebra structure
on a vector space $E$ containing ${Z}$ as a
subalgebra is isomorphic to such a unified product.

\begin{theorem}\label{classif}
Let $({Z},\cdot)$ be a Zinbiel algebra, $(E,\circ)$ be a Zinbiel algebra
containing ${Z}$ as a subalgebra in $E$. Then there exists an
extending datum $\Omega({Z}, V) =\bigl(\trl, \, \trr, \, \ppl, \,
\ppr, \, \omega, \, \ast \bigl)$ of ${Z}$
through a subspace $V$ of $E$ and an isomorphism of Zinbiel
algebras $E \cong {Z} \natural V$ that
stabilizes ${Z}$ and co-stabilizes $V$.
\end{theorem}

\begin{proof}
Let $p: E \to {Z}$ be the projection map such that $p(x) = x$, for all $x \in {Z}$.
Then $V := \rm{Ker}$$(p)$ is a complement of ${Z}$ in
$E$. We define the extending datum $\Omega({Z}, V)$ of ${Z}$ through $V$ by the following
formulas:
\begin{eqnarray*}
&& \trr : V \times {Z} \to {Z}, \quad
\,\,\, u \trr x= p \bigl(u\circ x\bigl), \\
&& \trl: V \times {Z} \to V,
\quad \,\, u \trl x = u\circ x - p \bigl(u\circ x\bigl),\\
&& \ppl \, : {Z} \times V \to {Z},
\quad x \ppl u = p\bigl(x\circ u\bigl), \\
&&\ppr \, : {Z} \times V \to V,
\quad x\ppr u = x\circ u - p\bigl(x\circ u\bigl),\\
&& \omega: V \times V \to {Z}, \quad \omega(u, v) = p \bigl(u\circ v\bigl),\\
&& \ast: V \times V \to V, \quad u\ast v =u\circ v- p \bigl(u\circ v\bigl).
\end{eqnarray*}
for all $x \in {Z}$, $u$, $v\in V$. Now, the map
$\varphi: {Z} \times V \to E$, $\varphi(x, u) := x + u$, is
a linear isomorphism between the direct product of vector spaces
${Z} \times V$ and the Zinbiel algebra $(E, \circ)$ with
the inverse given by $\varphi^{-1}(e) := \bigl(p(e), \, e -
p(e)\bigl)$, for all $e \in E$. Then there exists a unique
Zinbiel algebra structure on ${Z} \oplus V$ such that
$\varphi$ is an isomorphism of Zinbiel algebras and this unique
the bilinear map is given by:
$$
(x, u)\circ (y, v) := \varphi^{-1} \bigl(\varphi(x, u)\circ\varphi(y, v)\bigl).
$$
for all $x$, $y \in {Z}$ and $u$, $v\in V$. Using
\thref{1}, the proof will be finished if we prove that this
the bilinear map is the one defined by \equref{brackunif} associated to the
system $\bigl(\trl , \, \trr , \,
\ppl , \, \ppr , \, \omega , \ast
\bigl)$ constructed above. Indeed, for all $x$, $y \in
{Z}$, $u$, $v\in V$ we have:
\begin{eqnarray*}
&&(x, u)\circ (y, v) \\
&=& \varphi^{-1} \bigl(\varphi(x, u)\circ\varphi(y, v)\bigl) \\
&=& \varphi^{-1} \bigl(( x+u)\circ( y + v) \bigl) \\
&=& \varphi^{-1} \bigl( x\circ y +  x\circ v+ u\circ y+ u\circ v \bigl) \\
&=& \bigl( p ( x\circ y) +  p ( x\circ v) + p( u\circ y) + p( u\circ v) ,\\
&&x\circ v+ u\circ y+ u\circ v-  p ( x\circ v) - p( u\circ y) - p( u\circ v)\bigl) \\
&=& \bigl(  x\cdot y + x \ppl v + u\ppr y + \omega(u, v)
,x \trr v + u\trl y + u \ast v \bigl)
\end{eqnarray*}
as needed. Thus, $\varphi: {Z} \natural V \to E$ is an
isomorphism of Zinbiel algebras. Moreover, the following diagram is
commutative
\begin{eqnarray*}
\xymatrix {& {Z} \ar[r]^{i} \ar[d]_{id} & {{Z}
\natural V} \ar[r]^{q} \ar[d]^{\varphi} & V \ar[d]^{id}\\
& {Z} \ar[r]^{i} & {E}\ar[r]^{\pi } & V}
\end{eqnarray*}
where $\pi : E \to V$ is a natural projection of $E = {Z} + V$ on the
vector space $V$ and $q: {Z} \natural V \to V$, $q (x, u) := u$ is
the canonical projection. The proof is finished.
\end{proof}

Therefore, by the above Theorem \ref{classif}, for classifying of all Zinbiel
algebra structures on $E$ containing ${Z}$ as a Zinbiel
subalgebra, we only need to classify all unified products
${Z} \natural V$ associated to extending datum $\Omega({Z}, V) = \bigl(\trl, \,
\trr, \ppl, \, \ppr, \, \omega,\, \ast \bigl)$, for a given complement $V$ of ${Z}$ in $E$.
Next we construct the cohomological objects which will parameterize the classifying sets ${\rm Extd} \, (E,{Z})$ and respectively ${\rm Extd}' \, (E,{Z})$.

\begin{lemma} \label{lem:11}
Let $\Omega({Z}, V) = \bigl(\trl, \, \trr, \ppl, \, \ppr, \, \omega,\, \ast \bigl)$ and $\Omega'({Z}, V) = \bigl(\trl ', \, \trr ', \ppl ', \, \ppr ', \, \omega',\, \ast' \bigl)$ be two extending structures of ${Z}$ through $V$ and $ {Z} \natural V$, $ {Z} \natural ' V$ the associated unified products. Then there exists a bijection between the set of all homomorphisms of
Zinbiel algebras $\psi: {Z} \natural V \to {Z} \natural ' V$ which stabilizes ${Z}$ and the set of pairs $(r, s)$, where $r: V \to {Z}$, $s: V \to V$ are two
linear maps satisfying the following compatibility conditions for all $x \in {Z}$, $u$, $v \in V$:
\begin{enumerate}
\item[(M1)] $s(x \trr u) = x \trr ' s(u)$,
\item[(M2)] $s(u \trl x) = s(u) \trl ' x$,
\item[(M3)] $u \ppr x + r(u \trl x) = r(u)\cdot x+ s(u) \ppr ' x$,
\item[(M4)] $x \ppl u+ r(x \trr u) = x\cdot r(u) + x \ppl 's(u)$,
\item[(M5)] $s(u \ast v) = r(u) \trr ' s(v)+ s(u) \trl ' r(v) + s(u) \ast' s(v)$,
\item[(M6)] $\omega(u, \, v) + r(u \ast v) = r(u)\cdot r(v) + r(u)\ppl ' s(v) + s(u) \ppr ' r(v) + \omega' (s(u), s(v))$.
\end{enumerate}
The bijection is given as follows. For a pair $(r, s)$, the corresponding homomorphism of Zinbiel algebras $\psi =
\psi_{(r, s)}: {Z} \natural V \to {Z} \natural ' V$  is given by:
$$
\psi(x, u) = (x + r(u), s(u)).
$$
The homomorphism  $\psi = \psi_{(r, s)}$ is an  isomorphism if and only if
$s: V \to V$ is a bijective map and $\psi = \psi_{(r, s)}$ co-stabilizes $V$ if and only if $s= {id}_V$.
\end{lemma}

\begin{proof}
A linear map $\psi: {Z} \natural V \to {Z}
\natural ' V$ which stabilizes ${Z}$ is uniquely
determined by two linear maps $r: V \to {Z}$, $s: V \to
V$ such that $\psi(x, u) = (x + r(u), s(u))$, for all $x \in {Z}$, and $u \in V$.
Now we will prove that $\psi$ is a homomorphism of Zinbiel
algebras if and only if $(M1)$--$(M6)$
hold. We will check under what conditions the following equation holds
\begin{equation}\eqlabel{Liemap}
\psi \bigl((x, u)\circ (y, v) \bigl) = \psi(x, u)\circ' \psi(y, v).
\end{equation}
By direct computations, the left hand side of the above equation is equal to
\begin{eqnarray*}
&&\psi \bigl((x, u)\circ(y, v)\bigl) \\
 &=& \psi(x\cdot y + x \ppl v + u\ppr y + \omega(u, v),\  x \trr v + u\trl y + u \ast v)\\
 &=&\bigl(x\cdot y + x \ppl v + u\ppr y + \omega(u, v)+ r(x \trr v) +r(u\trl y)+r(u\ast v), \\
 &&\quad s(x \trr v) + s(u\trl y)+s(u\ast v)\bigl),
 \end{eqnarray*}
and the right hand side is equal to
\begin{eqnarray*}
 &&\psi(x, u)\circ' \psi(y, v)= (x+r(u),s(u))\circ (y+r(v), s(v))\\
 &=& \Big((x+r(u)) \cdot  (y+r(v))  + (x+r(u))\ppl' s(v) + s(u) \ppr' (y+r(v))+ \omega'(s(u), s(v)),\\
 && (x+r(u)) \trr ' s(v) + s(u) \trl ' (y+r(v)) +s(u)\ast' s(v)\Big).
 \end{eqnarray*}
Thus we obtain \equref{Liemap} holds if and only $(M1)$--$(M6)$ hold.

Assume that $s: V \to V$ is bijective linear map. Then $\psi$ is
an isomorphism of Zinbiel algebras with the inverse given by $
\psi_{(r, s)}^{-1}(y, v) = \bigl(y - r(s^{-1}(v)),
s^{-1}(v)\bigl)$, for all $y \in {Z}$ and $v \in V$.
Conversely, assume that $\psi$ is an isomorphism. One easily verify that $s$ is a bijection.
The last assertion is trivial. The proof is completed.
\end{proof}

\begin{definition}\delabel{echiaa}
Two extending structures $\Omega({Z}, V) =
\bigl(\trl, \, \trr, \ppl, \, \ppr, \, \omega,\, \ast \bigl)$ and
$\Omega'({Z}, V) = \bigl(\trl ', \,
\trr ', \ppl ', \, \ppr ', \, \omega',\, \ast' \bigl)$ are called \emph{equivalent} if there exists a pair $(r, s)$ of linear maps, where $r: V \to
{Z}$ and $s: V\to V$ satisfying the following conditions:
\begin{eqnarray*}
u \trl x &=& s^{-1} \bigl(s(u) \trl ' x\bigl),\\
x \trr u &=& s^{-1} \bigl(x \trr ' s(u)\bigl),\\
u \ppr x &=& r(u)\cdot x + s(u) \ppr ' x - r \circ s^{-1}
\bigl(s(u) \trl ' x\bigl), \\
x \ppl u &=& x\cdot r(u) + x \ppl ' s(u) - r
\circ s^{-1} \bigl(x \trr ' s(u)\bigl),\\
u \ast v &=& s^{-1} \bigl(r(u) \trr ' s(v) + s(u)
\trl ' r(v) + s(u) \ast' s(v)\bigl),\\
\omega(u,\, v) &=& r(u)\cdot r(v) + r(u) \ppl ' s(v) + s(u)
\ppr ' r(v) + \omega' \bigl(s(u), \, s(v)\bigl) - \\
&& r \circ s^{-1} \bigl(r(u) \trr ' s(v) + s(u)
\trl ' r(v) + s(u) \ast' s(v)\bigl),
\end{eqnarray*}
for all $x \in {Z}$, $u$, $v \in V$. We will denote it by $\Omega({Z}, V) \equiv \Omega'({Z}, V)$.
\end{definition}

Using the  Lemma \ref{lem:11}, we obtain that $\Omega({Z}, V)
\equiv \Omega'({Z}, V)$ if and only if there exists $\psi
: {Z} \natural V \to {Z} \natural' V$ an
isomorphism of Zinbiel algebras that stabilizes ${Z}$,
where ${Z} \natural V$ and ${Z} \natural' V$ are
the corresponding unified products. If $s=id_V$, then we obtain the isomorphism between two unified products that stabilize ${Z}$ and co-stabilize $V$ as follows:

\begin{definition}\delabel{echiaab}
Two extending structures $\Omega({Z}, V) =
\bigl(\trl, \, \trr, \ppl, \,\ppr, \, \omega,\, \ast \bigl)$ and
$\Omega'({Z}, V) = \bigl(\trl ', \,
\trr ', \ppl ', \, \ppr ', \, \omega',\ast' \bigl)$ are called \emph{cohomologous} if $\trl \, = \, \trl '$,
$\ppr \, = \,  \ppr '$ and there exists a
linear map $r: V \to {Z}$  satisfying the following conditions:
\begin{eqnarray*}
u \ppr x &=& r(u)\cdot x + u \ppr ' x - r(u \trl ' x), \\
x \ppl u &=& x\cdot r(u) + x \ppl ' u - r(x \trr ' u),\\
u \ast v &=& r(u) \trr ' v + u \trl ' r(v) + u \ast' v,\\
\omega(u,\, v) &=& r(u)\cdot r(v) + r(u) \ppl ' v + u
\ppr ' r(v) + \omega' (u, \, v)\\ &&-r \bigl(r(u)
\trr ' v + u \trl ' r(v) + u \ast' v\bigl),
\end{eqnarray*}
for all $x \in {Z}$, $u$, $v \in V$. We will denote it by $\Omega({Z}, V) \approx \Omega'({Z}, V)$.
\end{definition}

We denote by ${\mathcal Z} ({Z}, V)$ the set of extending structures $\Omega({Z}, V)$ .
It is easy to see that  $\equiv$  and $\approx$ are equivalence relations on the set   ${\mathcal Z} ({Z}, V)$.
By the above constructions and by Theorem  \ref{classif} and Lemma \ref{lem:11} we obtain the following result.

\begin{theorem}\thlabel{main1}
Let ${Z}$ be a Zinbiel algebra, $E$ a vector space that
contains ${Z}$ as a subspace and $V$ a complement of
${Z}$ in $E$. Then we get:

$(1)$ Denote ${\mathcal
H}{\mathcal E}^{2}_{} \, (V, \, {Z} ) :=
{\mathcal Z}  ({Z}, V)/ \equiv $, then the
map
$$
{\mathcal H}{\mathcal E}^{2}_{} \, (V, \, {Z}) \to {\rm Extd} \, (E, {Z}), \qquad
\overline{(\trl, \trr, \ppl, \,\ppr, \omega,\, \ast)} \mapsto {Z}\natural V
$$
is bijective, where $\overline{(\trl, \trr,
\ppl, \, \ppr, \omega,\, \ast)}$ is the
equivalence class of $(\trl, \trr,
\ppl, \, \ppr, \omega,\, \ast)$ via $\equiv$.

$(2)$ Denote ${\mathcal
H}{\mathcal C}^{2} \, (V, \, {Z} ) := {\mathcal Z} ({Z}, V)/ \approx $, then the map
$$
{\mathcal H}{\mathcal C}^{2} \, (V, \, {Z} ) \to {\rm Extd}' \, (E, {Z}), \qquad
\overline{\overline{(\trl, \trr,\ppl, \, \ppr, \omega,\, \ast)}} \mapsto{Z} \natural V
$$
is bijective, where $\overline{\overline{(\trl,
\trr, \ppl, \, \ppr, \omega,\, \ast)}}$ is the equivalence class of $(\trl,
\trr, \ppl, \, \ppr,  \omega,\, \ast)$ via $\approx$.
\end{theorem}

\section{Special cases of unified products}\selabel{cazurispeciale}

In this section, we show some special cases of unified products and extending structures.

\subsection{Crossed products and extension problem}

Now we give a first special case of the
unified product, namely the crossed product of Zinbiel algebras
which is related to the study of the extension problem.

Let ${Z}$ and ${W}$ be two given Zinbiel
algebras. The extension problem asks for the classification of all
extensions of ${W}$ by ${Z}$, i.e. all
Zinbiel algebras in $E$ that fit into an exact sequence
\begin{eqnarray} \label{extencros0}
\xymatrix{ 0 \ar[r] & {Z} \ar[r]^{i} & {E}
\ar[r]^{\pi} & {W} \ar[r] & 0 }.
\end{eqnarray}
The classification is up to an isomorphism of Zinbiel algebras
that stabilizes ${Z}$ and co-stabilizes ${W}$
and we denote by ${\mathcal E} {\mathcal P} ({W}, \,
{Z})$ the isomorphism classes of all extensions of
${W}$ by ${Z}$ up to this equivalence relation.

\begin{definition}
Let ${Z}$ and ${W}$ be two Zinbiel algebras.
Then crossed system consist of  linear maps
\begin{eqnarray*}
&&\ppr : {W} \times {Z} \to {Z}, \quad \ppl \, : {Z} \times {W} \to {Z}, \quad \omega: {W}\times {W} \to {Z},
\end{eqnarray*}
such that the direct sum space ${Z}\oplus {W}$ forms a Zinbiel algebra under the following  the bilinear map
\begin{equation*}
(x, u) \circ (y, v) := \big( x\cdot y + x \ppl v + u\ppr y + \omega(u, v), u \cdot v\big),
\end{equation*}
for all $ x,y \in {Z}$, $u,v \in {W}$.
This Zinbiel algebra is called crossed product of ${Z}$ and ${W}$ which will be  denoted   by ${Z}\,\#_\omega{W}$.
\end{definition}

\begin{theorem}
Let ${Z}$ and ${W}$  be two Zinbiel algebras.
Then ${Z}\,\#_\omega{W}$  is  a crossed product  if and only if the following conditions hold:
\begin{enumerate}
\item[(CS1)] $(x\ppl v)\cdot y=x\cdot (v \ppr y+y \ppl v)$,
\item[(CS1)] $(u \ppr x)\cdot y = u\ppr( x \cdot  y )+ u\ppr( y\cdot  x )$,
\item[(CS3)]
$\omega(u, v)\cdot x+ (u \cdot v )\ppr x
= u\ppr( v\ppr x + x \ppl v)$,
\item[(CS4)]$( x\cdot y )\ppl w  =x\cdot(y \ppl w+w\ppr y)$,
\item[(CS5)]
$( x \ppl v)\ppl w
=x\cdot\big(\omega(v, w)+\omega(w,v)\big)+x\ppl(v \cdot w +w \cdot v )$,
\item[(CS6)]
$( u\ppr x)\ppl w= u\ppr( x \ppl w+w\ppr x )$,
\item[(CS7)]$\omega(u, v)\ppl w+\omega( u \cdot v ,w ) = u\ppr\big(\omega(v, w)+\omega(w,v)\big)+ \omega(u,v \cdot w+ w \cdot v )$.
\end{enumerate}
\end{theorem}

Now we explain how to classify the extension problem.
Let ${Z}$ and ${W}$ be two Zinbiel algebras and we denote by ${\mathcal C}{\mathcal S}
\, ({W}, \, {Z} )$ the set of all triples
$(\ppr, \, \ppl, \, \omega)$ such that
${Z}\,\#_\omega{W}$ is a crossed product of Zinbiel algebras.
If $({Z}, \, {W}, \,
\ppr, \, \ppl, \, \omega)$ is a crossed system,
then the crossed product ${Z} \#_\omega \, {W}$ is an extension of
${W}$ by ${Z}$ via
\begin{eqnarray} \eqlabel{extencros001}
\xymatrix{ 0 \ar[r] & {Z} \ar[r]^{i} &
{Z}\,\#_\omega{W}
\ar[r]^{\pi} & {W} \ar[r] & 0 }
\end{eqnarray}
where $i (x) = (x, 0)$ and $\pi (x,u) = u$ are the canonical maps. Conversely, any extension $ {E}$ of ${W}$ by
${Z}$ is equivalent to a crossed product extension of the
form \equref{extencros001}. Thus, the classification of all
extensions of ${W}$ by ${Z}$ reduces to the
classification of all crossed products ${Z}\,\#_\omega{W}$ associated
to all crossed systems of Zinbiel algebras $({Z},
{W}, \ppr, \ppl, \omega)$.

\begin{definition}
Two crossed systems $(\ppr, \ppl, \omega)$ and $(\ppr', \ppl', \omega')$ of ${\mathcal C}{\mathcal S} \, ({W}, \,
{Z} )$ are \emph{cohomologous} and we denote this by
$(\ppr, \ppl, \omega) \approx (\ppr', \ppl', \omega')$ if there exists a linear map $r:
{W} \to {Z}$ such that:
\begin{eqnarray*}
u \ppr  x &=& u \ppr ' x  + r(u)\cdot x,\\
x \ppl u &=& x \ppl ' u + x\cdot r(u),\\
\omega(u, v) &=& \omega '(u, v) + r(u)\cdot r(v)- r (u\cdot v) + \, r (u) \ppl ' v + u \ppr ' r (v),
\end{eqnarray*}
for all $x \in {Z}$, $u$, $v \in {W}$.
\end{definition}

Note that $(\ppr,\ppl, \omega) \approx (\ppr', \ppl', \omega')$ if and only if there exists  an isomorphism of Zinbiel algebras that stabilizes
${Z}$ and co-stabilizes ${W}$. Thus we obtain the theoretical answer to the
extension problem.

\begin{corollary}\colabel{extprobra}
Let ${Z}$ and ${W}$ be two  Zinbiel
algebras. Then  $\approx$ is an equivalence relation on the set
${\mathcal C}{\mathcal S} \, ({W}, \, {Z} )$ of
all crossed systems and the map
$$
{\mathcal N} {\mathcal H}^2 ({W}, \, {Z}) :=
{\mathcal C}{\mathcal S} \, ({W}, \, {Z} )/
\approx \,\, \longrightarrow {\mathcal E} {\mathcal P}
({W}, \, {Z}), \qquad
\overline{(\ppr, \ppl, \omega)} \mapsto
{Z}\,\#_\omega{W}
$$
is a bijection between sets, where $\overline{(\ppr, \ppl, \omega)}$ is the equivalence class of
$(\ppr, \ppl, \omega)$ via $\approx$.
\end{corollary}

\subsection{Bicrossed products and the factorization problem for Zinbiel algebras}

\begin{definition}\label{mpLei}\cite{HD}
Let $({Z}, \, \cdot)$ and $({W}, \, \cdot)$ be two  Zinbiel
algebras. Then (${Z}$, ${W}$) is called a matched pair if there exists
four bilinear maps:
\begin{eqnarray*}
\trl : {W} \times {Z} \to
{W}, \qquad \trr : {W} \times {Z}
\to {Z},\\\, \qquad \ppl : {Z} \times
{W} \to {Z}, \qquad \ppr : {Z}
\times {W} \to {W},
\end{eqnarray*}
such that the direct sum space ${Z}\oplus {W}$ forms a Zinbiel algebra
under the following the bilinear map:
\begin{equation}\label{eq:matchedpair}
(x, u) \circ (y, v) := \big( x\cdot y + x \ppl v + u\ppr y,\,  x \trr v + u\trl y + u \cdot v \big)
\end{equation}
 where $x$, $y \in {Z}$ and $u$, $v \in {W}$. This Zinbiel algebra is
 called the \emph{bicrossed product}  of ${Z}$ and
${W}$.  We will denoted it by  ${Z} \, \bowtie {W}$.
\end{definition}

\begin{theorem}\label{mpLei}\label{thm:matchedpair}
Let ${Z}$ and ${W}$ be two Zinbiel algebras. Then ${Z} \, \bowtie {W}$  forms a Zinbiel algebra if
and only if the following compatibility conditions hold for all $x$, $y \in {Z}$, $u$, $v \in {W}$:
\begin{enumerate}
\item[(MP1)] $({W}, \trl, \trr)$ and $({Z}, \ppl, \ppr)$ are bimodules,

\item[(MP2)] $(u \ppr x)\cdot y +( u\trl x )\ppr y = u\ppr( x \cdot  y+ y\cdot  x )$,

\item[(MP3)]$( x\cdot y )\ppl w =(x\ppl v)\cdot y+(x \trr v)\ppr y $\\
                     $=x\cdot(y \ppl w+w\ppr y)+x\ppl(y \trr w+w\trl y)$

\item[(MP4)] $(x \ppl v) \trr w+(x\trr v)\cdot w = x\trr( v \cdot w+w\cdot v )$,

\item[(MP5)] $(u \cdot v)\trl x =(u\ppr x) \trr w+ (u\trl x)\cdot w$\\
$= u\trl( v\ppr x+x\ppl v)+u\cdot( v\trl x + x \trr v )$.
\end{enumerate}
\end{theorem}

The bicrossed product of two Zinbiel algebras is used to solve the {factorization problem}:
Let ${Z}$ and ${W}$ be two given Zinbiel algebras. Describe and classify all Zinbiel algebras ${E} $ that
factorize through ${Z}$ and ${W}$, i.e. ${E} $ contains ${Z}$ and ${W}$ as Zinbiel subalgebras such that ${E}  = {Z} \oplus {W}$ as vector spaces.
By Theorem \ref{classif} we can prove the following result.

\begin{corollary}\colabel{bicrfactor}
A Zinbiel algebra ${E} $ factorizes through ${Z}$ and
${W}$ if and only if there exists a matched pair of
Zinbiel algebras $({Z}, {W}, \, \trl,
\, \trr, \, \ppl, \, \ppr)$ such
that $ {E}  \cong {Z} \bowtie {W}$.
\end{corollary}

\begin{proof}
From the above Theorem \ref{thm:matchedpair}, it is easy to see that  any bicrossed product ${Z}
\bowtie {W}$ factorizes through ${Z}$ and ${W}$. Conversely, assume that ${E} $ factorizes through
${Z}$ and ${W}$. Let $p: {E}  \to {Z}$
be the linear projection of ${E} $ on ${Z}$, i.e. $p(x + u) := x$, for all $x\in {Z}$ and $u \in{W}$.
Now, we apply Theorem \ref{classif} for $V: = {\rm Ker}(p) = {W}$.
Since $V$ is a Zinbiel subalgebra of $E:= {E} $, the map $\omega = \omega_p$ constructed in the proof of
Theorem \ref{classif} is the trivial map as $[u, v] \in V = {\rm Ker}(p)$.
Thus, the matched pair $({Z},{W}, \, \trl, \,\trr, \, \ppl, \, \ppr)$ is given by:
\begin{eqnarray}
x \trr u := p \bigl(\left[x, \, u \right]\bigl), \qquad
x \trl u : = \left[x, \, u \right] - p \bigl(\left[x, \,
u \right]\bigl) \eqlabel{mpcanonic1} \\ u \ppl x :=
p\bigl([u, \, x]\bigl), \qquad u \ppr x := [u, \, x] -
p\bigl([u, \, x]\bigl)  \eqlabel{mpcanonic2}
\end{eqnarray}
for all $x \in {Z}$ and $u \in {W}$.
\end{proof}


\subsection{Classifying complements for Zinbiel algebras} \selabel{complements}

\begin{definition} \delabel{deformaplie}
Let $({Z}, \, {W}, \, \trr, \,
\trl, \, \ppl, \, \ppr)$ be a
matched pair of Zinbiel algebras. A linear map $r: {W}
\to {Z}$ is called a \emph{deformation map} of the
matched pair $({Z}, {W}, \trr,
\trl, \, \ppl, \, \ppr)$ if the
following compatibility holds for all $u$, $v \in {W}$:
\begin{equation}\eqlabel{factLie}
r\bigl(u\cdot v\bigl) \, - \, r(u)\cdot r(v) = u
\ppr r(v) + r(u) \ppl v - r \bigl(u
\trl r(v) + r(u) \trr v \, \bigl)
\end{equation}
\end{definition}

We will denote by ${\mathcal D}{\mathcal M} \, ({W},{Z})$ the set of all deformation
maps of the matched pair $({Z}, {W},
\trr, \trl, \ppl, \ppr)$.

\begin{theorem}\thlabel{deforLie}
Let ${Z}$ be a Zinbiel subalgebra of ${E} $,
${W}$ a given ${Z}$-complement of ${E} $ and $r:
{W} \to {Z}$ a deformation map of the associated
canonical matched pair $({Z}, {W},
\trr, \trl, \ppl, \ppr)$.

$(1)$ Let $f_{r}: {W} \to {E}  = {Z} \bowtie
{W}$ be the linear map defined for all $u \in{W}$ by:
$$f_{r}(u) = (r(u),\, u).$$
Then $\widetilde{{W}} : = {\rm Im}(f_{r})$ is a
${Z}$-complement of ${E} $.

$(2)$ Let ${W}_{r} := {W}$, as a vector space, with the new the bilinear map defined for all $u$, $v \in {W}$ by:
\begin{equation}\eqlabel{rLiedef}
u\cdot_{r}  v:= u\cdot v + u \trl r(v) + r(u)\trr v
\end{equation}
We call ${W}_{r}$ the $r$-deformation of ${W}$.
Furthermore, ${W}_{r} \cong \widetilde{{W}}$, as Zinbiel algebras.
\end{theorem}

\begin{proof}
$(1)$ First we will prove that $\widetilde{{W}}
= \{\bigl(r(u),\, u\bigl) ~|~ u \in {W}\}$ is a Zinbiel
subalgebra of ${Z} \bowtie {W} = {E}  $. In fact,
for all $u$, $v \in {W}$ we have:
\begin{eqnarray*}
&& (r(u), u)\circ (r(v), v)\\
&\stackrel{\equref{brackunif}}{=}& \Bigl(\underline{r(u)\cdot r(v)+ u \ppr r(v) + r(u) \ppl v}, \,
u\cdot v  + u \trl r(v) + r(u)\trr v \Bigl)\\
&\stackrel{\equref{factLie}}{=}& \Bigl(r (u\cdot v + u \ppr r(v)+ r(u)\trr v), \, u\cdot v + u
\trl r(v) + r(u)\trr v\Bigl)
\end{eqnarray*}
Therefore, $(r(u), u)\circ (r(v), v)\in \widetilde{{W}}$. Moreover, it is straightforward to see
that ${Z} \, \cap \,  \widetilde{{W}} = \{0\}$
and $(x,\, u) = \bigl(x - r(u), \, 0\bigl) + \bigl(r(u),\, x
\bigl) \in {Z} + \widetilde{{W}}$ for all $x \in {Z}$, $u \in {W}$. Here, we view ${Z}
\cong {Z} \times \{0\}$ as a subalgebra of ${Z}
\bowtie {W}$. Therefore, $\widetilde{{W}}$ is a
${Z}$-complement of ${E}  = {Z} \bowtie
{W}$.

$(2)$ We denote by $\widetilde{f_{r}} : {W} \to
\widetilde{{W}}$ the linear isomorphism induced by
$f_{r}$. We will prove that $\widetilde{f_{r}}$ is also a Zinbiel
algebra map if we consider on ${W}$ the the bilinear map given by
\equref{rLiedef}. Indeed, for all $u$, $v \in {W}$ we
have:
\begin{eqnarray*}
&&\widetilde{f_{r}}\bigl(u\cdot_r v\bigl)\\
&\stackrel{\equref{rLiedef}}{=}&
\widetilde{f_{r}}\bigl(u\cdot v + u \trl r(v) + r(u)\trr v\bigl)\\
&{=}& \Bigl(\underline{r \bigl(u\cdot v + u \trl r(v) +
r(u)\trr v\bigl)},\, u\cdot v +
u\trl r(v)+ r(u)\trr v\Bigl)\\
&\stackrel{\equref{factLie}}{=}& \Bigl(r(u)\ast r(v) + u \ppr r(v) + r((u) \ppl v,\, u\cdot v + u \ppr r(v)+ r(u)\trr v\Bigl)\\
&\stackrel{\equref{brackunif}}{=}&(r(u),\, u)\circ(r(v), \, v)
= \widetilde{f_{r}}(u)\cdot\widetilde{f_{r}}(v)
\end{eqnarray*}
Thus, ${W}_{r}$ is a Zinbiel algebra and the proof
is now finished.
\end{proof}

Conversely,  we have the following \thref{descrierecomlie} which show that
all ${Z}$-complements of ${E} $ are $r$-deformations.
The proof is similar as in \cite[Theorem 5.3]{AM3} so we omit the details.

\begin{theorem} \thlabel{descrierecomlie}
Let ${Z}$ be a Zinbiel subalgebra of ${E} $,
${W}$ a given ${Z}$-complement of ${E} $ with the
associated matched pair of Zinbiel algebras
$({Z}, {W}, \trr, \trl, ,
\ppl, \ppr)$. Then $\overline{{W}}$
is a ${Z}$-complement of ${E} $ if and only if there
exists an isomorphism of Zinbiel algebras $\overline{{W}}
\cong {W}_{r}$, for some deformation map $r: {W}
\to {Z}$ of the matched pair $({Z},{W})$.
\end{theorem}

\begin{definition}\delabel{equivLie}
Let $({Z}, {W}, \trr, \trl,
\ppl, \ppr)$ be a matched pair of Zinbiel
algebras. Two deformation maps $r$, $R: {W} \to
{Z}$ are called \emph{equivalent} and we denote this by
$r \sim R$ if there exists $\sigma: {W} \to {W}$
a $k$-linear automorphism of ${W}$ such that for all $x$,
$y\in {W}$:
\begin{eqnarray*}\eqlabel{equivLiemaps}
&&\sigma \bigl(u\cdot v\bigl) -\sigma(u)\cdot \sigma(v)\\
&=& \sigma(u) \trl R \bigl(\sigma(v)\bigl) + R
\bigl(\sigma(u)\bigl) \ppr \sigma(v) - \sigma\bigl(u
\trl r(v)\bigl) - \sigma \bigl(r(u) \trr v\bigl)
\end{eqnarray*}
\end{definition}

\begin{theorem}\thlabel{clasformelorLie}
Let ${Z}$ be a Zinbiel subalgebra of ${E} $,
${W}$ a ${Z}$-complement of ${E} $ and
$({Z}, {W}, \trr, \trl,
\ppl, \ppr)$ the associated canonical matched
pair. Then $\sim$ is an equivalence relation on the set ${\mathcal D}{\mathcal M} \, ( {W}, {Z})$.
If we define ${\mathcal H}{\mathcal A}^{2} ({W}, {Z}) := \, {\mathcal D}{\mathcal M} \,
({W}, {Z})/ \sim$, then the map
$$
{\mathcal H}{\mathcal A}^{2} ({W}, {Z})\rightarrow {\mathcal F} ({Z}, \, {E} ),
\quad
\overline{r} \mapsto {W}_{r}
$$
is a bijection between ${\mathcal H}{\mathcal A}^{2}
({W}, {Z})$ and the
isomorphism classes of all ${Z}$-complements of ${E} $. In
particular, the factorization index of ${Z}$ in ${E} $ is
computed by the formula:
$$
[{E}  : {Z}] = | {\mathcal H}{\mathcal A}^{2}
({W}, {Z})|
$$
\end{theorem}

\begin{proof}
Follows from \thref{descrierecomlie} taking into account the fact
that two deformation maps $r$ and $R$ are equivalent in the sense
of \deref{equivLie} if and only if the corresponding Zinbiel
algebras ${W}_r$ and ${W}_R$ are isomorphic.
Thus we obtain the result.
\end{proof}

\section{Flag extending structures of Zinbiel algebras}\selabel{exemple}

In this section, we will study extending structures of a given Zinbiel algebra ${Z}$ by $V$ in the case when $V$ is a one dimensional vector space.
This will be  called flag extending structures.

\begin{definition} \delabel{tehnicaa}
Let ${Z}$ be a Zinbiel algebra. A \emph{flag datum} of ${Z}$ is a $5$-tuple $(x_{0}, \, k_0,
\mu, \,D,\, T)$, where $x_{0} \in {Z}$, $k_0 \in k$, and $\mu: {Z} \to k$, $D, T:{Z} \to {Z}$ are linear maps satisfying the
following compatibility conditions:
\begin{enumerate}
\item[(F1)] $\mu(x\cdot y) =\mu( y )\mu(x) - \mu(y\cdot x),$
\item[(F2)] $\mu( D( x ) + T( x )) =0,$
\item[(F3)] $D(x\cdot y) = D( x )y + \mu( x )D( y ) - D(y\cdot x),$
\item[(F4)] $T(x\cdot y) = T(x)y,$
\item[(F5)] $T^{2}(x) = 2x\cdot x_{0} + 2k_{0}T(x),$
\item[(F6)] $D^{2}(x) = D( T(x)) + \mu(x)x_{0}- T( D(x)),$
\item[(F7)] $T( D(x)) = x_{0}\cdot x + k_{0}D(x),$
\item[(F8)] $T(x_{0}) = 2D(x_{0}) + k_{0}x_{0} , \quad 0 = 2\mu (x_{0}) + k^{2}_{0}.$
\end{enumerate}
We denote by ${\mathcal F}({Z})$ the set of all flag datums of ${Z}$.
\end{definition}

\begin{proposition}\label{flag-prop1}
Let ${Z}$ be a Zinbiel algebra and $V$ a vector space of
dimension $1$ with a basis $\{u\}$. Then there exists a bijection
between the set ${\mathcal Z}  \, ({Z}, V)$
of all extending structures of ${Z}$ through $V$
and ${\mathcal F} \, ({Z})$.
The Zinbiel extending
structure $\Omega({Z}, V)  = \bigl(\trl, \,\trr, \ppl,\, \ppr, \omega, \ast\bigl)$ corresponding to $(x_0, k_0, \mu, D, T)\in {\mathcal F}, ({Z})$ is given by:
\begin{eqnarray}
&&u \trl x = \mu (x) u,  \quad x\trr u =0, \\
&& u \ppr x =D(x), \quad x \ppl u = T(x),\\
&&\omega(u, u) = x_0, \quad  u \ast u = k_0 \, u,
\end{eqnarray}
The corresponding unified product ${Z}\natural V$ is given by
\begin{equation}
(x, u) \circ (y, u) := \Big( x\cdot y + T(x) + D(y) + x_0, \, \mu(y) u+  k_0 \, u \Big)
\end{equation}
and the Zinbiel algebra $A\natural \{u\}$ is generated by the following relations
\begin{equation}
u\circ u = x_0 + k_0 \, u, \quad x\circ u = T(x), \quad u\circ x =D(x) + \mu (x) \, u.
\end{equation}
\end{proposition}

The proof of the above Proposition \ref{flag-prop1} is omitted since it is a special case of  \thref{1} in last sections.
Now we shall give the classification of all flag extending structures of  ${Z}$ based on \thref{main1}.

\begin{theorem}
Let ${Z}$ be an algebra of codimension $1$ in the vector space $V$. Then:
$\operatorname{Extd}(V, {Z}) \cong{\mathcal{A}{{\mathcal{H}}}}^{2}(k, {Z}) \cong \mathcal{F}({Z}) / \equiv$, where $\equiv$ is the equivalence relation on the set $\mathcal{F}({Z})$
defined as follows: $\left(\mu,  D, T, x_{0}, k_{0}\right) \equiv$ $\left(\mu^{\prime},  D^{\prime}, T^{\prime}, x_{0}^{\prime}, k_{0}^{\prime}\right)$
if and only if $\mu(x)=\mu^{\prime}(x)$ and there exists a pair
$(r, s) $, where $r: V \to {Z}$, $s: V \to V$, $u \mapsto qu,$ $q \in k, u \in V$ are two
linear maps such that:
$$
\begin{aligned}
D(x) &=q D^{\prime}(x)+r(u) \cdot x - \mu(x) r(u), \\
T(x) &=q T^{\prime}(x)+x r(u), \\
x_{0} &=q^{2} x_{0}^{\prime}+r^{2}(u)-k_{0}r(u) + q T^{\prime}(r(u))+q D^{\prime}(r(u)), \\
k_{0} &=q k_{0}^{\prime}+D^{\prime}(r(u)),
\end{aligned}
$$
for all $x \in {Z}$.
The bijection between $\mathcal{F} {Z}) / \equiv$ and $\operatorname{Extd}(V, {Z})$ is given by:
$$
\overline{\left(\mu,  D, T, x_{0}, k_{0}\right)} \mapsto {Z} \ltimes_{\left(\mu,  D, T, x_{0}, k_{0}\right)}\{u\}
$$
where $\overline{\left(\mu,  D, T, x_{0}, k_{0}\right)}$ is the equivalence class of $\left(\mu,  D, T, x_{0}, k_{0}\right)$ via the relation $\equiv$.
\end{theorem}

\begin{example}
Let ${Z}$ be 3-dimensional Zinbiel algebras with a basis $\{e_1, e_2, e_3\}$, the bilinear maps are given as follows:\\
\noindent
$(A_{1})$ $$\quad e_1\cdot e_1=e_3,$$
$(A_{2})$ $$\quad e_1\cdot  e_1=e_3,\quad  e_2\cdot  e_2=e_3,$$
$(A_{3})$ $$\quad e_1\cdot  e_2=\half e_3,\quad e_2\cdot  e_1=-\half e_3,$$
$(A_{4})$ $$\quad e_2\cdot  e_1= e_3,$$
$(A_{5})$ $$\quad e_1\cdot  e_1=e_3,\quad e_1\cdot  e_2=e_3,\quad e_2\cdot  e_2=\lambda e_3(\lambda \neq 0),$$
$(A_{6})$ $$\quad e_1\cdot  e_1=e_2,\quad e_1\cdot  e_2=\half e_3,\quad e_2\cdot  e_1=e_3.$$
\end{example}

Now we compute the flag datum for some special cases.
To simplify the calculation, first we  let $x_0=0, k_0=0, T=0$, then the conditions of flag datum is transformed to
\begin{eqnarray*}
&&\mu( x\cdot y ) =\mu( y )\mu(x) - \mu( y\cdot x ),\\
&&\mu( D( x ) ) = 0,\\
&&D( x\cdot y ) = D( x )\cdot y + \mu( x )D( y ) - D( y\cdot x ),\\
&&D^{2}(x) = 0.
\end{eqnarray*}

Denote by
\[
D\left( {\begin{array}{l}
 e_1 \\
 e_2 \\
 e_3 \\
 \end{array}} \right) =
\left( {{\begin{array}{*{20}c}
 {a_{11} }  & a_{12}  & {a_{13} } \\
  a_{21}  & a_{22}  & a_{23} \\
 {a_{31} }  & a_{32}  & {a_{33} } \\
\end{array} }} \right)
\left( {\begin{array}{l}
 e_1 \\
 e_2 \\
 e_3 \\
 \end{array}} \right),
\]
$$\mu(e_{1})=\mu_{1},\,\mu(e_{2})=\mu_{2},\,\mu(e_{3})=\mu_{3}.$$

For $(A_{1})$, we have $\mu_2=0,\mu_1\neq 0, \mu_{3} = \half \mu_{1}^{2}$,
\[
D_1=
\left( {{\begin{array}{*{20}c}
0  &  0  & 0 \\
a_{21}  & 0  & -\frac{2a_{21}}{\mu_{1}} \\
0  & 0  & 0\\
\end{array} }} \right).
\]

For $(A_{2})$,  we have $\mu_2=0,\mu_1\neq 0, \mu_{3} = \half \mu_{1}^{2},$
then $D_2=0$.

For $(A_{3})$,  we have the following two cases: if $\mu_{1}=0, \mu_{2} \neq 0, \mu_{3} \neq 0$,
\[
D_{31} =
\left( {{\begin{array}{*{20}c}
  0   & 0  & 0 \\
 a_{21}   & 0  & 0 \\
a_{31}   & 0  & 0 \\
\end{array} }} \right);
\]
if $\mu_{2} = 0, \mu_{1} \neq 0, \mu_{3} \neq 0$,
\[
D_{32}=
\left( {{\begin{array}{*{20}c}
0  &  a_{12}  & 0 \\
0& 0  & 0 \\
0& a_{32}  & 0\\
\end{array} }} \right).
\]

For $(A_{4})$, we have $\mu_{3} =\mu_{1}\mu_{2}$,
\[
D_{41}=
\left( {{\begin{array}{*{20}c}
0  & 0 & 0 \\
 a_{21}  & 0  & a_{23} \\
0  & 0  & 0\\
\end{array} }} \right), \quad
D_{42}=
\left( {{\begin{array}{*{20}c}
0  &  a_{12}  & 0 \\
0  & 0  & 0 \\
0  & \mu_{2}a_{12}  & 0\\
\end{array} }} \right).
\]

For $(A_{5})$, we have $\mu_1=\mu_2=\mu_3=0$,
\[
D_5=
\left( {{\begin{array}{*{20}c}
0  & 0 & a_{13} \\
0  & 0  & a_{2 3} \\
0  & 0  & 0\\
\end{array} }} \right).
\]

For $(A_{6})$,  we have $\mu_1\neq 0, \mu_{2} = \half \mu_{1}^{2},\mu_{3} = \frac{1}{3}\mu_{1}^{3},$  $D_6= 0.$

Thus we obtain 4-dimensional  Zinbiel algebras with a basis $\{e_1, e_2, e_3, u\}$, the bilinear maps are given as follows:\\
$(DA_{1})$ $$\quad e_1\cdot e_1=e_3, \quad u\cdot x_1=\mu_1u,  \quad u\cdot x_2=a_{21}x_1-\textstyle{\frac{2a_{21}}{\mu_{1}}}x_3, \quad u\cdot x_3= \half  \mu_1^2u;$$
$(DA_{2})$ $$\quad e_1\cdot  e_1=e_3,\quad  e_2\cdot  e_2=e_3, \quad u\cdot x_1= \mu_1u, \quad u\cdot x_3= \half  \mu_1^2 u;$$
$(DA_{3})$
$$(1)  \quad e_1\cdot  e_2=\half e_3,\quad e_2\cdot  e_1=-\half e_3,\quad u\cdot x_2=a_{21}x_1+{\mu_{2}}u, \quad u\cdot x_3= a_{31}x_1+\mu_3u;$$
$$(2) \quad e_1\cdot  e_2=\half e_3,\quad e_2\cdot  e_1=-\half e_3,\quad u\cdot x_1=a_{12}x_2+{\mu_{1}}u, \quad u\cdot x_3= a_{32}x_2+\mu_3u;$$
$(DA_{4})$
$$(1) \quad e_2\cdot  e_1= e_3,\quad u\cdot x_2=a_{21}x_1+a_{23}x_3;\qquad\qquad\qquad\qquad\qquad\quad\qquad\qquad\qquad$$
$$(2) \quad e_2\cdot  e_1= e_3,\quad u\cdot x_1=a_{12}x_2,\quad u\cdot x_3=\mu_2a_{12}x_2;\qquad\qquad\qquad\qquad\qquad\qquad$$
$(DA_{5})$
$$e_1\cdot  e_1=e_3,\quad e_1\cdot  e_2=e_3,\quad e_2\cdot  e_2=\lambda e_3(\lambda \neq 0),\quad u\cdot x_1=a_{13}x_3,\quad u\cdot x_2=a_{23}x_3;$$
$(DA_{6})$ $$\quad e_1\cdot  e_1=e_2,\quad e_1\cdot  e_2=\half e_3,\quad e_2\cdot  e_1=e_3, \quad u\cdot x_2=  \half  \mu_1^2  u, \quad u\cdot x_3= \third  \mu_1^3 u.$$

Similarly, if we  let $x_0=0, k_0=0, D=0$, then the conditions of flag datum is transformed to
\begin{eqnarray*}
&&\mu( x\cdot y ) =\mu( y )\mu(x) - \mu( y\cdot x ),\\
&&\mu( T( x ) ) = 0,\\
&&T( x\cdot y ) = T( x )\cdot y ,\\
&&T^{2}(x) = 0.
\end{eqnarray*}

Denote by
\[
T\left( {\begin{array}{l}
 e_1 \\
 e_2 \\
 e_3 \\
 \end{array}} \right) =
\left( {{\begin{array}{*{20}c}
 {b_{11} }  & b_{12}  & {b_{13} } \\
  b_{21}  & b_{22}  & b_{23} \\
 {b_{31} }  & b_{32}  & {b_{33} } \\
\end{array} }} \right)\quad
\left( {\begin{array}{l}
 e_1 \\
 e_2 \\
 e_3 \\
 \end{array}} \right),
\]
$$\mu(e_{1})=\mu_{1},\, \mu(e_{2})=\mu_{2},\, \mu(e_{3})=\mu_{3}.$$

For $(A_{1})$, we have $\mu_{3} = \half \mu_{1}^{2},$
\[
\mathop T\nolimits_{11}=
\left( {{\begin{array}{*{20}c}
0  &  0  & 0 \\
b_{21}  & 0  & -\frac{2}{\mu_{1}}b_{21} \\
0  & 0  & 0\\
\end{array} }} \right),\quad
\mathop T\nolimits_{12}=
\left( {{\begin{array}{*{20}c}
0  &  b_{12}  & -\frac{2\mu_{2}}{\mu_{1}^{2}}b_{12} \\
0  & 0  & 0 \\
0  & 0  & 0\\
\end{array} }} \right).
\]

For $(A_{2})$, we have $\mu_{3} = \half \mu_{1}^{2} = \half \mu_{2}^{2}, \mu_{2} = \mu_{1},$
\[
\mathop T\nolimits_{12} =
\left( {{\begin{array}{*{20}c}
0   & 0  & 0 \\
  b_{21}   & 0  &  -\frac{2}{\mu_{1}}b_{21} \\
 0   & 0  & 0 \\
\end{array} }} \right),\quad
\mathop T\nolimits_{22} =
\left( {{\begin{array}{*{20}c}
0   & b_{12}  & -\frac{2}{\mu_{1}}b_{12} \\
 0   & 0  &  0 \\
 0   & 0  & 0 \\
\end{array} }} \right).
\]

For $(A_{3})$, we have $\mu_{2}\mu_{1}=0$.
If $\mu_{1}=0, \mu_{2} \neq 0, \mu_{3} \neq 0$, we obtain
\[
\mathop T\nolimits_{31} =
\left( {{\begin{array}{*{20}c}
 0   & 0  & 0 \\
 b_{21}   & 0  &  0 \\
 0   & 0  & 0 \\
\end{array} }} \right),\quad
\mathop T\nolimits_{32} =
\left( {{\begin{array}{*{20}c}
 0   & b_{12}  & -\frac{\mu_{2}}{\mu_{3}}b_{12} \\
 0   & 0  &  0 \\
 0   & 0  & 0 \\
\end{array} }} \right),
\]
If $\mu_{2}=0, \mu_{1} \neq 0, \mu_{3} \neq 0$,
we obtain
\[
\mathop T\nolimits_{33} =
\left( {{\begin{array}{*{20}c}
 0   & 0  & 0  \\
 b_{21}   & 0  &   -\frac{\mu_{1}}{\mu_{3}}b_{21} \\
 0   & 0  & 0 \\
\end{array} }} \right),\quad
\mathop T\nolimits_{34} =
\left( {{\begin{array}{*{20}c}
 0   & b_{12}  & 0  \\
 0   & 0  &  0 \\
 0   & 0  & 0 \\
\end{array} }} \right).
\]

For $(A_{4})$, we have $\mu_{3} = \mu_{2}\mu_{1}$, then
\[
\mathop T\nolimits_{41} =
\left( {{\begin{array}{*{20}c}
0   & 0  & 0 \\
 b_{21}   & 0  &  -\frac{1}{\mu_{2}}b_{21} \\
 0   & 0  & 0 \\
\end{array} }} \right),\quad
\mathop T\nolimits_{42} =
\left( {{\begin{array}{*{20}c}
b_{11}   & b_{12}  & -\frac{b_{11}}{\mu_{2}} - \frac{b_{12}}{\mu_{1}} \\
 0   & 0  &  0 \\
 0   & 0  & 0 \\
\end{array} }} \right).
\]

For $(A_{5})$,  we have $\mu_1=\mu_2=\mu_3=0$,
\[
T_{51}=
\left( {{\begin{array}{*{20}c}
0  & 0 & 0 \\
0  & 0  & b_{2 3} \\
0  & 0  & 0\\
\end{array} }} \right),
\]
or $\mu_{2} = \half \mu_{1}, \mu_{3} = \half \mu_{1}^{2}$, then $T_{52}= 0$.

For $(A_{6})$,  we have $\mu_{2} = \half \mu_{1}^{2}, \mu_{3} = \frac{1}{3}\mu_{1}^{3},$  then $T_6 = 0.$

Thus we obtain 4-dimensional  Zinbiel algebras with a basis $\{e_1, e_2, e_3, u\}$, the bilinear maps are given as follows:\\
$(TA_{1})$
\begin{eqnarray*}
(1)&&e_1\cdot e_1=e_3, \quad u\circ x_1 = \mu_1 u, \quad u\circ x_2 = \mu_2 u,\quad u\circ x_3 =\half \mu_{1}^2 u,\qquad\qquad\qquad\\
&&x_2\circ u = b_{21}x_1-\textstyle{\frac{2}{\mu_{1}}}b_{21}x_3,\\
(2)&&e_1\cdot e_1=e_3, \quad u\circ x_1 = \mu_1 u,\quad u\circ x_2 = \mu_2 u, \quad u\circ x_3 =\half \mu_{1}^2 u,\qquad\qquad\qquad\\
&&x_1\circ u = b_{12}x_2-\textstyle{\frac{2\mu_{2}}{\mu_{1}^{2}}}b_{12} x_3;
\end{eqnarray*}
$(TA_{2})$
\begin{eqnarray*}
(1) && e_1\cdot  e_1=e_3,\quad  e_2\cdot  e_2=e_3, \quad u\circ x_1 = \mu_1 u,\quad u\circ x_2 = \mu_1 u,\quad u\circ x_3 =\half \mu_{1}^2 u,\\
&&x_2\circ u = b_{21}x_1-\textstyle{\frac{2}{\mu_{1}}}b_{21}x_3,\\
(2) && e_1\cdot  e_1=e_3,\quad  e_2\cdot  e_2=e_3, \quad u\circ x_1 = \mu_1 u,\quad u\circ x_2 = \mu_1 u,\quad u\circ x_3 =\half \mu_{1}^2 u,\\
&&x_1\circ u = b_{12}x_2-\textstyle{\frac{2}{\mu_{1}}}b_{12} x_3;
\end{eqnarray*}
$(TA_{3})$
\begin{eqnarray*}
(1)&&e_1\cdot  e_2=\half e_3,\quad e_2\cdot  e_1=-\half e_3, \quad u\circ x_2 = \mu_2 u, \quad u\circ x_3= \mu_3 u,\qquad\qquad\qquad\\
&&x_2\circ u = b_{21}x_1,\\
(2)&&e_1\cdot  e_2=\half e_3,\quad e_2\cdot  e_1=-\half e_3, \quad u\circ x_2 = \mu_2 u, \quad u\circ x_3= \mu_3 u,\qquad\qquad\qquad\\
&&x_1\circ u = b_{12}x_2-\textstyle{\frac{\mu_{2}}{\mu_{3}}}b_{12} x_3,\\
(3)&&e_1\cdot  e_2=\half e_3,\quad e_2\cdot  e_1=-\half e_3, \quad u\circ x_1 = \mu_1 u,\quad u\circ x_3 = \mu_3 u, \qquad\qquad\qquad\\
&&x_2\circ u = b_{21}x_1-\textstyle{\frac{\mu_{1}}{\mu_{3}}}b_{21}x_3,\\
(4)&&e_1\cdot  e_2=\half e_3,\quad e_2\cdot  e_1=-\half e_3, \quad u\circ x_1 = \mu_1 u,\quad u\circ x_3 = \mu_3 u, \qquad\qquad\qquad\\
&&x_1\circ u = b_{12}x_2;
\end{eqnarray*}
$(TA_{4})$
\begin{eqnarray*}
(1) &&e_2\cdot  e_1= e_3, \quad u\circ x_1 = \mu_1 u,\quad u\circ x_2 = \mu_2 u,\quad u\circ x_3 = \mu_{1}\mu_2 u, \qquad\qquad\qquad\\
&&x_2\circ u = b_{21}x_1-\textstyle{\frac{1}{\mu_{1}}}b_{21}x_3,\\
(2) &&e_2\cdot  e_1= e_3,\quad u\circ x_1 = \mu_1 u,\quad u\circ x_2 = \mu_2 u,\quad u\circ x_3 =\mu_{1}\mu_2 u, \qquad\qquad\qquad\\
&&x_1\circ u = b_{11}x_1+b_{12}x_2-\textstyle{\left(\frac{b_{11}}{\mu_{2}}+\frac{b_{12}}{\mu_{1}}\right)} x_3;
\end{eqnarray*}
$(TA_{5})$
\begin{eqnarray*}
(1) && e_1\cdot  e_1=e_3,\quad e_1\cdot  e_2=e_3,\quad e_2\cdot  e_2=\lambda e_3(\lambda \neq 0),\quad x_2\circ u = b_{23}x_3,\qquad\qquad\\
(2) && e_1\cdot  e_1=e_3,\quad e_1\cdot  e_2=e_3,\quad e_2\cdot  e_2=\lambda e_3(\lambda \neq 0),\\
 &&u\circ x_1 = \mu_1 u,\quad u\circ x_2 = \half\mu_1 u,\quad u\circ x_3 = \half\mu_{1}^2 u;
\end{eqnarray*}
$(TA_{6})$
\begin{eqnarray*}
 &&e_1\cdot  e_1=e_2,\quad e_1\cdot  e_2=\half e_3,\quad e_2\cdot  e_1=e_3\\
  &&u\circ x_1 = \mu_1 u,\quad u\circ x_2 = \half\mu_1^2 u,\quad u\circ x_3 = \third\mu_{1}^3 u.
\end{eqnarray*}

We remark that the above constructed 4-dimensional Zinbiel algebras are not new,
since a general  classification of 4-dimensional Zinbiel algebras was given in \cite{AOK}.

\section{Conclusions and problems}
In this paper,  we developed the theory of extending structures for Zinbiel algebras.
On the other side, it was proved in \cite{DT,T22} that each finite-dimensional Zinbiel algebra is nilpotent.
Thus it is natural to study the extending structures for nilpotent Zinbiel algebras and its relationship with  central extensions considered in \cite{CKK,KAM}.
This is left for future investigations.

\end{document}